\documentclass[a4paper,11pt]{article}

\usepackage{amsmath,amsthm,amssymb}
\usepackage[top=3cm,bottom=3cm,left=3cm,right=3cm]{geometry} 
\usepackage[T1]{fontenc} 
\usepackage{url}
\usepackage{hyperref}

\DeclareMathOperator{\SO}{SO}

\DeclareMathOperator{\U}{U}

\newcommand{\R}{\mathbb R}

\newcommand{\Z}{\mathbb Z}

\newcommand{\so}{\mathfrak{so}}

\renewcommand{\geq}{\geqslant}

\theoremstyle{plain}
	\newtheorem{theorem}{Theorem}

\theoremstyle{definition}

\theoremstyle{plain}
	\newtheorem*{theorem*}{Theorem}
	\newtheorem*{proposition*}{Proposition}
	\newtheorem*{lemma*}{Lemma}
	\newtheorem*{corollary*}{Corollary}
	\newtheorem*{conjecture*}{Conjecture}
\theoremstyle{definition}
	\newtheorem*{definition*}{Definition}
	\newtheorem*{remark*}{Remark}
	\newtheorem*{remarks*}{Remarks}

\usepackage{authblk}

\begin{document}
	
\title{Symplectic domination}
\author{Joel Fine\footnote{D\'epartement de math\'ematiques, Universit\'e libre de Bruxelles, Belgium. \texttt{\href{mailto:joel.fine@ulb.ac.be}{joel.fine@ulb.ac.be}
}}~~and Dmitri Panov\footnote{Department of Mathematics, King's College London, United Kingdom. \texttt{\href{mailto:dmitri.panov@kcl.ac.uk}{dmitri.panov@kcl.ac.uk}}}}

\date{ }

\maketitle

The aim of this short note is to prove the following theorem.

\begin{theorem}\label{symplectic-domination}
	Let $M$ be a compact oriented manifold of even dimension. There exists a map of positive degree $f \colon S \to M$ from a compact symplectic manifold $S$ of the same dimension. 
\end{theorem}

This result says, in some sense, that there are ``a lot'' of symplectic manifolds. This fits with the philosophy behind a folklore conjecture in symplectic topology, stated as Conjecture~6.1 in the article \cite{eliashberg} of Eliashberg. The conjecture asserts that if $X$ is a compact manifold of dimension $2n \geq 6$, which admits an almost complex structure and a cohomology class $\kappa \in H^2(X,\R)$ with $\kappa^n \neq 0$, then $M$ carries a symplectic structure.

Theorem~\ref{symplectic-domination} follows rather quickly from two deep results (stated as Theorems~\ref{ontaneda} and~\ref{skd} below). The first is a spectacular construction by Ontaneda of Riemannian manifolds with tightly pinched negative curvatures. 
\begin{theorem}[Ontaneda]\label{ontaneda}
	Let $M$ be a compact oriented manifold and $\epsilon >0$. There exists a degree~one map $f \colon N \to M$ from a compact oriented Riemannian manifold $N$ of the same dimension, with sectional curvatures in the interval $[-1-\epsilon, -1]$.
\end{theorem}

This is the main result of a lengthy preprint \cite{ontaneda1}, which was subsequently broken up into a series of articles for publication \cite{ontaneda2,ontaneda3,ontaneda4,ontaneda5,ontaneda6,ontaneda7,ontaneda8}. The pinched manifolds constructed by Ontaneda are smoothings of singular negatively curved manifolds constructed by Charney and Davis using a procedure called strict hyperbolisation \cite{CH}. This in turn builds on the hyperbolisation of polyhedra by Gromov \cite{Gromov}.

For our purposes, the important consequence of the curvature pinching is that the twistor space of $N$ carries a natural symplectic form.  We recall that the twistor space $Z \to N$ of an oriented Riemannian manifold $N$ is the bundle of compatible almost complex structures on the tangent spaces. I.e.~the fibre of $Z$ over $x \in N$ is the set of all linear orthogonal complex structures on $T_xN$ which induce the given orientation. The fibres are homogeneous spaces, identified with $F = \SO(2n)/\U(n)$. The symplectic form on $Z$ is provided by a construction due to Reznikov (which is, in fact, a special case of Weinstein's ``fat bundles'' \cite{weinstein}).

\begin{theorem}[Reznikov \cite{reznikov}]
Let $N$ be an oriented even-dimensional Riemannian manifold with twistor space $Z$. There is a natural closed 2-form $\omega$ on $Z$ with integral cohomology class $[\omega] \in H^2(Z,\Z)$ which is symplectic when restricted to each fibre of $Z \to N$. Moreover, there is a positive number $\epsilon >0$, depending only on the dimension of $N$, such that if the sectional curvatures of $N$ lie in the interval $[-1-\epsilon, -1]$ then $\omega$ is symplectic.\end{theorem}

Since this is central to the proof of Theorem~\ref{symplectic-domination}, we explain briefly how the construction goes. The key to the existence of an integral closed 2-form is that the model fibre $F = \SO(2n)/\U(n)$ of twistor space is a homogeneous integral symplectic manifold. In other words, there is a principle $S^1$-bundle $P_F \to F$ with a connection $A_F$ whose curvature is a symplectic form on $F$; moreover $P_F$ carries an action of $\SO(2n)$ covering the action on $F$ and leaving $A_F$  invariant. 

This can be seen via the theory of integral coadjoint orbits (see, for example, \cite{kirillov}), but it is also simple to describe it explicitly. Consider the action of $G \in \U(2n)$ on $\SO(2n)\times S^1$ by
\[
G \cdot (R, e^{i\theta}) = (RG^{-1}, \det(G) e^{i\theta})
\]
We denote the quotient by $P_F=\SO(2n)\times_{\U(n)} S^1$. The $S^1$-action on $\SO(2n)\times S^1$ given by multiplication on the second factor commutes with the diagonal action of $\U(n)$ and so descends to $P_F$ making it a principle $S^1$-bundle over $F$. Moreover, the $\SO(2n)$-action on $\SO(2n)\times S^1$ given by multiplication on the left of the first factor also commutes with the diagonal action of $\U(n)$, and so descends to an $\SO(2n)$-action on $P_F$, where it covers the $\SO(2n)$-action on $F$. Finally, to see the connection consider the derivative $\so(2n) \to T_pP_F$ of the $\SO(2n)$-action at a point $p \in P_F$. It is  transverse to the $S^1$-orbit through $p$, the image gives the horizontal distribution defining the $\SO(2n)$-invariant connection $A_F$.

We now return to the twistor space $Z \to N$ and carry out this construction on every fibre. The result is a principle $S^1$-bundle $P \to Z$ fitting together the fibrewise bundles $P_F \to F$. Moreover, the connection $A_F$ gives a fibrewise connection in $Z$. To promote this to a genuine connection in all of $P \to Z$ we must specify the horizontal distribution transverse to the fibres of $Z \to N$; but this is precisely what the Levi-Civita connection does. This gives a connection $A$ in $P\to Z$ whose curvature determines a closed integral 2-form $\omega$ which is symplectic on each fibre. 

One can now ask for $\omega$ to be symplectic, which becomes a curvature inequality for the Riemannian metric on $N$. Reznikov observed that this inequality is satisfied by hyperbolic space and so, by openness, it is also satisfied by all negatively curved metrics which are sufficiently pinched. In the case $\dim N =4$, the article \cite{fine-panov} gives the full curvature inequality explicitly.

The next step in the proof is to invoke another deep theorem, namely Donaldson's result on symplectic hypersurfaces. 

\begin{theorem}[Donaldson \cite{donaldson}]\label{skd}
	Let $(Z,\omega)$ be a compact symplectic manifold with $[\omega]$ an integral cohomology class. There exists a symplectic submanifold $S$ of codimension~2, with $[S]$ Poincaré dual to a positive multiple $k[\omega]$ of the symplectic class.
\end{theorem}

\begin{proof}[Proof of Theorem~\ref{symplectic-domination}]

By Ontaneda's Theorem it suffices to prove the result for all compact oriented even-dimensional Riemannian manifolds $N$ with sectional curvatures pinched arbitrarily close to $-1$. 

Suppose first that $\dim N=4$. In this case, the twistor space $Z \to N$ has fibres $S^2$. By Reznikov's result we know that there is an integral symplectic form on $Z$ for which the twistor fibres are symplectic. Now let $S\subset Z$ be a Donaldson hypersurface, with $[S] = k[\omega]$ for $k>0$. The twistor projection restricts to a smooth map $f \colon S \to N$ and we claim the degree of this map is positive. To prove this write $[F]$ for the homology class of a fibre of $Z \to N$. The intersection number $[S]\cdot[F] = k \int_F\omega$ is positive since it is a positive multiple of the symplectic area of $F$. It follows that $f$ is surjective. Now Sard's theorem implies the existence of a point $x \in N$ which is not a critical value of $f$. This means that $S$ meets the fibre $F_x$ over $x$ transversely. The local degree of $f$ at each point of $F_x \cap S$ is equal to the local intersection of $F_x$ and $S$ at that point, hence the degree of $f$ equals $[S]\cdot [F]$ which we have just seen is positive.

In higher dimensions the argument is similar. When $\dim N = 2n$ the twistor space has dimension $n(n+1)$ and the fibre has dimension $n(n-1)$. We start as before with a Donaldson hypersurface $S_1 \subset Z$, with $[S_1]$ Poincar\'e dual to $k_1[\omega]$. We apply Donaldson's theorem again, this time to $(S_1, \omega_{S_1})$, to obtain a symplectic submanifold $S_2 \subset S_1 \subset Z$, where $S_2$ has codimension~4 in $Z$ with $[S_2]$ Poincaré dual to $k_2[\omega]^2$. We continue in this way, producing a chain $S_d \subset S_{d-1} \subset \cdots \subset S_1 \subset Z$ of symplectic submanifolds where $d = n(n-1)/2$. Each $S_j$ is a symplectic submanifold of $Z$ of codimension $2j$ and so $S_d$ has complimentary dimension to a fibre of $Z \to N$. Moreover, $[S_d]$ is Poincaré dual to $k [\omega]^d$ for some $k>0$. It follows that $[S_d]\cdot [F] = 
k \int_F \omega^d$ which is positive since it is a positive multiple of the symplectic volume of the fibre. From here the same argument as before shows that the twistor projection $f \colon S_d \to N$ has positive degree. 
\end{proof}

We close with a remark, that the symplectic manifolds $(S,
\omega)$ produced in the proof of Theorem~\ref{symplectic-domination} are of ``general type'' in the sense that $c_1(S) = - p[\omega]$ where $p>0$. This follows from adjunction and the fact, proved in \cite{fine-panov2}, that when $dim N=2n$, the symplectic structures on the twistor space satisfy $c_1(Z) = (n-2) [\omega]$. 

\subsubsection*{Acknowledgements} We would like to thank Igor Belegradek and Anton Petrunin for discussions. JF was supported by ERC consolidator grant 646649 ``SymplecticEinstein''.

\end{document}